\documentclass[12pt]{article}

\usepackage{epsfig}
\usepackage{amsfonts}
\usepackage{amssymb}
\usepackage{amstext}
\usepackage{amsmath}
\usepackage{xspace}
\usepackage{theorem}
\usepackage{graphicx}
\usepackage{lineno}
\usepackage{hyperref}

\newcommand{\junk}[1]{}

\makeatletter
\newenvironment{proof}{{\bf Proof:  }}{\hfill\rule{2mm}{2mm}}
\newenvironment{proofof}[1]{{\bf Proof of #1:  }}{\hfill\rule{2mm}{2mm}}

\newtheorem{theorem}{Theorem}
\newtheorem{lemma}[theorem]{Lemma}
\newtheorem{proposition}[theorem]{Proposition}

\newtheorem{corollary}[theorem]{Corollary}

\newtheorem{conjecture}[theorem]{Conjecture}
\newtheorem{question}[theorem]{Question}

\newcommand{\Z}{\ensuremath{\mathbb Z}}
\newcommand{\N}{\ensuremath{\mathbb N}}

\usepackage{bm}
\newcommand{\vect}[1]{\bm{#1}}

\title{Polychromatic Colorings on the Integers}

 \author{
 Maria Axenovich\thanks{Karlsruhe Institute of Technology, Karlsruhe, Germany, \texttt{maria.aksenovich@kit.edu}.} 
 \and John Goldwasser\thanks{ West Virginia University, Morgantown, WV, USA, \texttt{jgoldwas@math.wvu.edu}.}
 \and Bernard Lidick\'y\thanks{Iowa State University, Ames, IA, USA, \texttt{lidicky@iastate,edu}. Supported by NSF grant DMS-1600390.}
 \and Ryan R. Martin \thanks{Iowa State University, Ames, IA, USA, \texttt{rymartin@iastate.edu}. Research supported in part by Simons Foundation Collaboration Grant (\#353292, to R.R. Martin).}
 \and David Offner\thanks{Westminster College, New Wilmington, PA, USA, \texttt{offnerde@westminster.edu}. Supported by Westminster College McCandless Research Award.}
 \and John Talbot\thanks{University College London, London, UK, \texttt{j.talbot@ucl.ac.uk}.}
 \and Michael Young\thanks{Iowa State University, Ames, IA, USA, \texttt{myoung@iastate.edu}.}
 }

\begin{document}

\maketitle

\begin{abstract}
We show that for any set $S\subseteq \mathbb{Z}$, $|S|=4$ there exists a 3-coloring of $\mathbb{Z}$ in which every translate of $S$ receives all three colors. This implies that $S$ has a codensity of at most $1/3$, proving a conjecture of Newman [D. J. Newman,  Complements of finite sets of integers, \textit{Michigan Math. J.} 14 (1967) 481--486]. 
We also consider related questions in $\mathbb{Z}^d$, $d\geq 2$.
\end{abstract}

\section{Introduction}
Throughout the paper, let $G$ denote an arbitrary abelian group.  Given $S,T \subseteq G$, $n \in G$, define $S+T = \{s+t: s\in S, t\in T\}$ and $n+S  = \{n\} + S$. Any set of the form $n+S$ is called a \textit{translate} of $S$. Given a subset $S$ of $G$, a coloring of the elements of $G$ is \textit{$S$-polychromatic} if every translate of $S$ contains an element of each color.  Define the \textit{polychromatic number} of $S$, denoted $p_G(S)$, to be the largest number of colors allowing an $S$-polychromatic coloring of the elements of $G$.  We just write $p(S)$ when the choice of $G$ is clear from context. 

We begin with some elementary observations that will be used repeatedly.  First, if $S'$ is a subset of $S$, then any $S'$-polychromatic coloring is also an $S$-polychromatic coloring, and thus $p(S') \le p(S)$. Also, if $S'=n+S$ is a translate of $S$, then $S'$ and $S$ have the same set of translates, so $p(S') = p(S)$.

We are primarily concerned with the setting where $G= \Z$ and $S$ is finite. If $S \subseteq \Z$ has cardinality 1 or 2, $p(S)=|S|$. For $|S|=3$, $p(S)$ can be 2 or 3.  For example, if $S=\{0,1,5\}$ then every translate of $S$ contains three elements which are each in different congruence classes $\pmod{3}$.  Thus a 3-coloring of the integers where each congruence class (mod 3) is colored with a different color is $S$-polychromatic, and $p(\{0,1,5\})=3$. However $p(\{0,1,3\})=2$.  To see that $p(\{0,1,3\}) \neq 3$, let $\chi$ be a 3-coloring of $\Z$ with $\chi(0)$, $\chi(1)$, and $\chi(3)$ all different. Some element $s \in \{0,1,3\}$ has $\chi(s) = \chi(2)$, and there is a translate of $\{0,1,3\}$ that contains both $s$ and 2, so the coloring is not polychromatic.  Our main result concerns the polychromatic numbers of sets with cardinality 4.

\begin{theorem}\label{ps34}
If $S \subseteq \Z$ and $|S|=4$, then $p(S) \ge 3$.
\end{theorem}

The proof of Theorem~\ref{ps34} is given in Section~\ref{ss4}.  For larger sets $S$, Alon, 
K{\v{r}}{\'\i}{\v{z}}, and Ne\v{s}et\v{r}il~\cite{AKN95} proved that $p(S) \ge \frac{(1+o(1))|S|}{3\ln|S|}$, while there exists some set $S$ where $p(S) \le \frac{(1+o(1))|S|}{\ln|S|}$. Subsequently, Harris and Srinivasan~\cite{HS16} established a tight asymptotic lower bound on polychromatic numbers.

\begin{theorem}[\cite{AKN95}, \cite{HS16}]\label{Harris}
For a finite set $S \subseteq \Z$, $p(S) \ge \frac{(1+o(1))|S|}{\ln|S|}$. Moreover,  there exists some set $S$ where $p(S) \le \frac{(1+o(1))|S|}{\ln|S|}$.
\end{theorem}

One motivation for studying polychromatic numbers is that they provide bounds for Tur\'an type problems (see for example \cite{AKS07}, \cite{Off09}, \cite{OS11}). Call $T \subseteq G$ a \textit{blocking set} for $S$ if $G \setminus T$ contains no translate of $S$, i.e. if for all $n \in G$, $n+S\nsubseteq G\setminus T$. A Tur\'an type problem asks for the smallest blocking set for a given set $S$. In the case where $S$ is finite and $G=\Z$, any blocking set is countably infinite, so we ask how small the density of a blocking set can be. Following the notation of Newman~\cite{New67}, (he worked in the setting of the natural numbers, but the definitions are equivalent), define for any set $T \subseteq \Z$ its \textit{upper density} $\overline{d}(T)$ and \textit{lower density} $\underline{d}(T)$ as 
\[\overline{d}(T) = \limsup_{n\rightarrow \infty} \frac{|T \cap [-n,n]|}{2n+1} \hspace{.25in} \text{ and } \hspace{.25in} \underline{d}(T) = \liminf_{n\rightarrow \infty} \frac{|T \cap [-n,n]|}{2n+1}. \]
If $\overline{d}(T)=\underline{d}(T)$, we call this quantity the \textit{density} of $T$ and denote it by $d(T)$. Define $\alpha(S)$ to be a measure of how small the density of a blocking set for $S$ can be. Let
\[\alpha(S) = \inf\{d(T):\text{ $T$ is a blocking set for $S$ and $d(T)$ exists}\}.\]

In Section~\ref{polcov}, we describe the relationship between polychromatic colorings and blocking sets, and prove Lemma~\ref{csps}.
\begin{lemma}\label{csps}
For any finite set $S \subseteq \Z$, $\alpha(S) \le 1/p(S)$.
\end{lemma}

One of the main consequences of Theorem~\ref{ps34} concerns covering densities of sets of integers. Given a set $S \subseteq G$, we say $T\subseteq G$ is a \textit{complement set} for $S$ if $S+T = G$.  We say $S$ \textit{tiles} $G$ \textit{by translation} if it has a complement set $T$ such that if $s_1, s_2 \in S$, $t_1, t_2 \in T$, then $s_1 + t_1=s_2+t_2$ implies $s_1 = s_2$ and $t_1 = t_2$. We call such a complement set $T$ a \textit{tiling complement set} for $S$.  Note that the set $S$ tiles $G$ by translation if all the translates $S + t$, $t \in T$, are disjoint and every $n \in G$ is an element of some translate $S+t$. In this paper we only consider tilings by translation, so if $S$ tiles $G$ by translation with tiling complement set $T$ we will simply say $S$ tiles $G$ and write $G=S \oplus T$.

Again, our primary interest will be the case where $G= \Z$ and $S$ is finite. For example, if $S=\{0,1,5\}$, then $S$ tiles $\Z$ with complement set $T=\{3n:n\in \Z\}$. However $S=\{0,1,3\}$ does not tile $\Z$. Newman~\cite{New77} proved necessary and sufficient conditions for a finite set $S$ to tile $\Z$ if $|S|$ is a power of a prime. 

\begin{theorem}[Newman \cite{New77}]\label{newtile77}
Let $S=\{s_1, \ldots, s_k\}$ be distinct integers with $|S|=p^\alpha$ where $p$ is prime and $\alpha$ is a positive integer.  For $1\le i<j\le k$ let $p^{e_{ij}}$ be the highest power of $p$ that divides $s_i-s_j$.  Then $S$ tiles $\Z$ if and only if $|\{e_{ij}: 1 \le i <j \le k\}| \le \alpha$.
\end{theorem} 

Later Coven and Meyerowitz~\cite{CM99} gave necessary and sufficient conditions for $S$ to tile $\Z$ when $|S|= p_1^{\alpha_1}p_2^{\alpha_2}$, where $p_1$ and $p_2$ are primes. The general question is still open. Kolountzakis and Matolcsi~\cite{KM09} and Amiot~\cite{Ami16} have published recent work motivated by what are called rhythmic tilings in music.

If a finite set $S$ tiles $\Z$, it has a complement set of density $1/|S|$. Following Newman~\cite{New67}, we define the \textit{codensity} of a set $S$, denoted $c(S)$, as a measure of how small the density of a complement set can be. Let
\[c(S) = \inf\{d(T): S+T=\Z \text{ and $d(T)$ exists}\}.\]


We are interested in the largest codensities for sets of a given cardinality.  Define 
\[c_k = \sup_{\{S:|S|=k\}}c(S).\]  

An example of a complement set for $\{0,1,3\}$ is $\{t \in \Z: t \equiv 0 \text{ or } 1 \pmod{5}\}$, so $c(\{0,1,3\}) \le 2/5$. The following theorem and conjecture on $c_4$ are due to Newman.

\begin{theorem}[Newman ~\cite{New67}]\label{newt}
\ 
\begin{itemize}
\item $c(\{0,1,3\}) = 2/5$.
\item $c_3 = 2/5$.
\item $c(\{0,1,2,4\}) = 1/3$.
\end{itemize}
\end{theorem}

\begin{conjecture}[Newman~\cite{New67}]\label{c4conj}
 $c_4 = 1/3$. 
 \end{conjecture}
 
Conjecture~\ref{c4conj} is stated and attributed to Newman by Weinstein~\cite{Wei76}, who proved that $c_4 < .339934$. Based on a computer search, Bollob\'as, Janson, and Riordan \cite{BJR11} confirmed Newman's conjecture for sets with diameter at most 22, where the \textit{diameter} of a nonempty finite set of integers is defined to be the difference between the largest and smallest elements in the set.  They also conjectured that $c_5 = 3/11$ and $c_6  = 1/4$ (See Remark 5.6 and Question 5.7 in \cite{BJR11}.  Note they use different notation).

In Section~\ref{polcov} we prove the following lemma relating blocking sets and complement sets.
\begin{lemma}\label{alphc}
For any finite set $S \subseteq \Z$, $c(S)= \alpha(S)$. 
\end{lemma}

Theorem~\ref{ps34}, along with Lemmas~\ref{csps} and \ref{alphc}, suffice to resolve Conjecture~\ref{c4conj}.

\begin{theorem}\label{c413}
 $c_4=1/3$.
\end{theorem}

\begin{proof}
Theorem~\ref{newt} implies $c(\{0,1,2,4\})=1/3$, so it remains to show that for any other set $S$ with cardinality four, $c(S) \le 1/3$.  Let $S \subseteq \Z$ have four elements.  Then Theorem~\ref{ps34} implies that $p(S) \ge 3$, and by Lemmas~\ref{csps} and \ref{alphc},
\[c(S) = \alpha(S) \le 1/p(S) \le 1/3.\]
\end{proof}

In Subsection~\ref{tiliff} we consider the relationship between polychromatic colorings and tilings.  The main result is Theorem~\ref{polytile}, which states that a set $S$ tiles an abelian group $G$ by translation if and only if $p(S)=|S|$. 

Finally, in Section~\ref{z2} we turn our attention to polychromatic numbers and tilings for finite sets in $\Z^d$.  We begin by proving in Theorem~\ref{hsd} that the bound of Theorem~\ref{Harris} applies to subsets of  $\Z^d$. We then show that if a set of points in $\Z^d$ is collinear, determining its polychromatic number is equivalent to determining the polychromatic number of a specific projection of this set into $\Z$. Theorem~\ref{polytile} implies that a set $S$ tiles $\Z^d$ if and only if $p_{\Z_d}(S)=|S|$, so we use this to restate some well-known results on tilings of $\Z^d$ by finite sets in the language of polychromatic colorings.   We conclude by applying these results to determine polychromatic numbers of sets with cardinality 3 and 4 in $\Z^d$.

\section{Sets of Cardinality Four}\label{ss4}
In this section we prove that every set of four integers has polychromatic number at least 3. We begin by stating some general lemmas that reduce the problem of finding an $S$-polychromatic coloring of $\Z$ to finding an $S$-polychromatic coloring of $\Z_m= \{0,1,\ldots,m-1\}$ for a specific choice of $m$.


\begin{lemma}\label{hom}
If $G$ and $H$ are abelian groups and $\phi: G \to H$ is a homomorphism, then for all $S \subseteq G$, \[p_G(S) \ge p_H(\phi(S)).\]
\end{lemma}

\begin{proof}
Let $S \subseteq G$ and let $\chi'$ be a $\phi(S)$-polychromatic coloring of $H$ with $p_H(\phi(S))$ colors. Define the coloring $\chi$ on $G$ such that $\chi(g) = \chi'(\phi(g))$.  Consider a translate $g+S$ of the set $S$. Since
\[\chi(g+S) = \chi'(\phi(g+S)) = \chi'(\phi(g) + \phi(S)),\]
 and $\phi(g) + \phi(S) \subseteq H$ is a translate of $\phi(S)$, $\chi(g+S)$ contains all $p_H(\phi(S))$ colors, and $\chi$ is $S$-polychromatic.
\end{proof}

\begin{corollary}\label{iso}
If $G$ and $H$ are abelian groups and $\phi: G \to H$ is an isomorphism, then for all $S \subseteq G$, \[p_G(S) = p_H(\phi(S)).\]
\end{corollary}

\begin{lemma}\label{ambient}
If $H$ is a subgroup of an abelian group $G$, and $S \subseteq H$, then $p_H(S) = p_G(S)$.
\end{lemma}

\begin{proof}
Since $H$ is a subset of $G$, $p_H(S) \ge p_G(S)$.  To prove the other inequality, suppose $\chi'$ is an $S$-polychromatic coloring of $H$ with $p_H(S)$ colors.  Let $V \subseteq G$ be a set containing exactly one element from each coset of $H$.  For every $g \in G$, there is a unique $h \in H$ and $v \in V$ such that $g = h+v$.  Define a coloring $\chi$ of $G$ such that $\chi(g)= \chi'(h)$, i.e. the summand $v$ is ignored.  We show that $\chi$ is $S$-polychromatic.  Given $g\in G$ with $g=h+v$ for some $h \in H$, $v \in V$, consider the translate $g+S$, and note
\[\chi(g+S) = \chi(h+v+S) = \chi'(h+S).\]
Since $h+S \subseteq H$ is a translate of $S$, $\chi(g+S)$ contains all $p_H(\phi(S))$ colors, and $\chi$ is $S$-polychromatic.
\end{proof}


\begin{lemma}\label{reduc}
Suppose $a,b,c, k \in \Z$ with $0<a<b<c$,  $k \ge 1$. Let $S = \{0,ka,kb,kc\}$, $S_1 = \{0, a, b, c\}$, and $S_2 = \{0, b-a, b, 2b-a\}$.  Then 
\begin{enumerate}
\item[(i)] $p_{\Z}(S) = p_{\Z}(S_1)$.
\item[(ii)] If $q \in \N$, then $p_{\Z}(S_1) \ge p_{\Z_q}(S_1)$.
\item[(iii)] If $m=c-a+b$, then $p_{\Z_m}(S_1) = p_{\Z_m}(S_2)$.
\item[(iv)] If $q \in \N$, with $\gcd(k,q) = 1$, then  $p_{\Z_q}(S) = p_{\Z_q}(S_1)$.
\end{enumerate}
\end{lemma}
\begin{proof}
\begin{enumerate}
\item[(i)] Define $\phi: \Z \to k\Z$ such that $\phi(n) = kn$.  Then $\phi$ is an isomorphism where $\phi(S_1) = S$, so Corollary~\ref{iso} implies  $p_{k\Z}(S) = p_\Z(S_1)$. Since $k\Z$ is a subgroup of $\Z$ and $S \subseteq k\Z$, Lemma~\ref{ambient} implies $p_{k\Z}(S) = p_\Z(S)$.  By combining these equations, we conclude $p_\Z(S) = p_{k\Z}(S) = p_\Z(S_1)$.
\item[(ii)] This part follows from Lemma~\ref{hom} using the homomorphism $\phi:\Z \to \Z_q$ where $\phi(n) = n \pmod{q}$.
\item[(iii)] In $\Z_m$, with addition $\pmod{m}$,  $S_2 = S_1 + (b-a)$. Thus in $\Z_m$, $S_1$ and $S_2$ are translates of each other and have the same polychromatic number.
\item[(iv)]  Define $\phi:\Z_q \to \Z_q$ so that $\phi(n) = kn$. Since $\gcd(k,q)=1$, $\phi$ is an isomorphism.  Since $\phi(S_1) = S$, Corollary 10 implies $p_{\Z_q}(S) = p_{\Z_q}(S_1)$.
\end{enumerate}
\end{proof}

\begin{proofof}{Theorem~\ref{ps34}}
Let $S \subseteq \Z$ have cardinality four.  Since all translates of $S$ have the same polychromatic number, we may assume that 0 is the smallest element of $S$, and by Lemma~\ref{reduc}, Part (i), it suffices to prove the theorem in the case that $S=\{0,a,b,c\}$ with $0<a<b<c$ and $\gcd(a,b,c)=1$.

It is possible, though tedious, to prove the entire theorem by hand.  Thus in the interest of simplifying the exposition, we verified using a computer search that for every $S$ with diameter at most 288 there exists an $S$-polychromatic 3-coloring of $\Z_q$ for some $q$ depending on $S$.  The code for this search has been included as an ancillary file with the preprint of this paper at arxiv.org/abs/1704.00042.  By Lemma~\ref{reduc}, Part (ii), this gives a periodic $S$-polychromatic 3-coloring of $\Z$.  Hence we suppose that $c \geq 289$. 

For the remainder of the proof, let $m=c-a+b$. By Lemma~\ref{reduc}, Parts (ii) and (iii), it suffices to show that we can 3-color $\Z_m= \{0,1,\ldots,m-1\}$ so that the translates of $\{0,b-a,b,2b-a\}$ are polychromatic.  So for the remainder of the proof we assume $S=\{0,b-a,b,2b-a\}$ and seek an $S$-polychromatic 3-coloring of $\Z_m$. The key observation regarding $S$ is that it contains two repeated differences: $b-a$ and $b$.

Define $d_1=\gcd(b,m)$ and $d_2=\gcd(b-a,m)$. Since $1=\gcd(a,b,c)=\gcd(b-a,b,c-a+b) = \gcd(b-a, b, m)$, we know $\gcd(d_1,d_2)=1$. We distinguish two main cases. In the first case, which we call ``single cycle,'' we assume $\min\{d_1,d_2\}=1$ and give a coloring of $\Z_m$. In the second case, which we call ``multiple cycle,'' we assume $\min\{d_1,d_2\}>1$ and partition $\Z_m$ into multiple cycles of length $m/d_i$ for one of the choices of $i$.  We then give a rule for coloring each cycle.

\textbf{Main case 1 (Single cycle):}  Suppose $\min\{d_1,d_2\}=1$. Without loss of generality, assume $d_1=1$ (if not, then simply switch all occurences of $b$ and $b-a$ in the argument below).  Let $2\leq g \le m-2$ satisfy $gb \equiv b-a \pmod{m}$, so that $S = \{0,bg,b,b(g+1)\}$.
Applying Lemma~\ref{reduc}, Part (iv), with $q=m$ and $k=b$, we can instead work with $S = \{0,g,1,g+1\} = \{0,1,g,g+1\}$.
 
We may assume that $g\leq m/2$, as otherwise we could work with the translate $(m-g)+S =\{0,1, m-g, m-g+1\}$. Let $s$ be the smallest multiple of 3 such that $g>\lceil m/s\rceil$. We consider four subcases:  The first two are (1a) $g=2$, $3$, or $4$ and (1b) $5 \le g<2\lfloor m/s \rfloor$. In the remaining subcases (1c) and (1d), $2\lfloor m/s\rfloor\leq g \leq \lceil m/(s-3)\rceil$. For $m >8$, if $2\lfloor m/s\rfloor\leq g\leq m/2$ then $s >3$, and for $m > 44$, if $2\lfloor m/s\rfloor\leq g \leq \lceil m/(s-3)\rceil$ then $s<9$.  Since $m > c \ge 289 > 44$, we can assume $s=6$, so $2\lfloor m/6\rfloor\leq g \leq \lceil m/3\rceil$. This implies $m=3g+k$ where $-2 \le k \le 5$ and there are two further subcases to consider, depending on the residue class of $m$ modulo 6: (1c)  $m=3g-2$, $3g-1$, $3g+1$, $3g+2$, $3g+4$, or $3g+5$, and (1d) $m=3g$ or $3g+3$.

\textbf{Subcase (1a):} Suppose $g=2$, $3$, or $4$. Then $S = \{0,1,2,3\}$,  $\{0,1,3,4\}$, or $\{0,1,4,5\}$, respectively.  In Subcase (1c) we will construct $S$-polychromatic 3-colorings of $\Z_m$ for each of these sets.

\textbf{Subcase (1b):}  Suppose $5 \le g<2\lfloor m/s \rfloor$. Then split $\Z_m$ into $s$ intervals as equally as possible (i.e. of lengths $\lfloor m/s\rfloor$ and $\lceil m/s \rceil$) and color these intervals $010101\ldots$, followed by $121212\ldots$, then $202020\ldots$, repeating $s/3$ times. Since $\lceil m/s\rceil<g<2\lfloor m/s \rfloor$, any translate of $S'$ where the pairs $\{0,1\}$ and $\{g,g+1\}$ lie in different intervals gets all three colors. If one of the pairs $\{0,1\}$ or $\{g,g+1\}$ straddles two consecutive intervals, this pair may get only the single color common to these two intervals, but then the other pair lies fully inside a third interval which is colored with the remaining two colors.

\textbf{Subcase (1c):} Suppose  $m=3g-2$, $3g-1$, $3g+1$, $3g+2$, $3g+4$, or $3g+5$. In this case we know that $m\not\equiv 0 \pmod{3}$ so we can apply Lemma~\ref{reduc}, Part (iv), with $q=m$ and $k=3$, and instead work with one of the sets in $\mathcal{S} =  \{\{0,2,3,5\}, \{0,1,3,4\}, \{0,1,2,3\}, \{0,3,4,7\}, \{0,3,5,8\}\}$. For example, if $m=3g-2$, then multiplying by 3, $S$ is transformed into $\{0,3,3g,3g+3\} \equiv \{0,2,3,5\}$, while if $m=3g+4$, then multiplying by 3, $S$ is transformed into $\{0,3,3g,3g+3\} \equiv \{0,3,-4,-1\}$, which is a translate of $\{0,3,4,7\}$.

Thus we have reduced the problem to finding an $S$-polychromatic 3-coloring of $\Z_m$ for each of the sets $S\in \mathcal{S}$. For each $S\in \mathcal{S}$, in Table~\ref{periodrr1} we list one interval of length $r$ and one of length $r+1$ obtained by adding an initial 0 to the other interval. We  also include an interval for $\{0,1,4,5\}$ to cover Subcase (1a).  Each of the intervals has the property that concatenating the intervals of length $r$ and $r+1$ in any way results in an $S$-polychromatic coloring for the corresponding set.  One can check this by hand, using the fact that in each case, a translate of $S\in \mathcal{S}$ intesects at most two consecutive intervals. Hence if $m$ can be expressed as a positive integer combination of $r$ and $r+1$, $m=hr + k(r+1)$, we can obtain an $S$-polychromatic coloring with period $m$. For $r=3,6,7,9$, by the 2-coin Frobenius problem, $m$ can be expressed as a positive integer combination of $r$ and $r+1$ for any $m$ greater than $r^2-r-1\leq 71<289$.

%


\begin{table}
\begin{center}
\begin{tabular}{c|c|l|l}
$S$ &$r$ &period $r$ & period $r+1$ \\ \hline
$\{0,2,3,5\}  $&6& 001122& 0001122\\
$\{0,1,3,4\} $&6& 001212& 0001212\\ 
$\{0,1,2,3\}$ &3& 012& 0012\\
$\{0,3,4,7\}$&9& 000111222& 0000111222\\
$\{0,3,5,8\}$&9& 000111222& 0000111222\\ 
$\{0,1,4,5\}  $&7& 0001212& 00001212\\
\end{tabular}
\end{center}
\caption{One interval of a periodic coloring for sets in Subcases (1a) and (1c).}
\label{periodrr1}
\end{table}

\textbf{Subcase (1d):} Suppose $m=3g$ or $3g+3$. If $g\not\equiv 0 \pmod{3}$ then simply color $\Z_m$ with the pattern $0120120\ldots012$. If $g\equiv 0\pmod{3}$ and $m=3g$, color $\Z_m$ in 3 equal intervals, each of length $g$: $012012\ldots012$ followed by $120120\ldots120$ followed by $201201\ldots201$. Finally, if $g\equiv0 \pmod{3}$ and $m=3g+3$ we color $\Z_m$ in 3 equal intervals, each of length $g+1$: $012012\ldots0120$ followed by $201201\ldots2012$ followed by $120120\ldots1201$.

\textbf{Main case 2 (Multiple cycles):} Suppose $\min\{d_1,d_2\}>1$. Since $d_1$ and $d_2$ are relatively prime, at most one of them can be a multiple of 3.  Choose the smallest of these numbers that is not a multiple of 3, and as in the single cycle case, without loss of generality assume it is $d_1$.

Let $e_1=m/d_1$ and $e_2=m/d_2$. For $0 \le i <d_1$, let 
\[C_i = \{(b-a)i +bj\pmod{m}: 0 \le j < e_1\}.\] 
Since 
\[\Z_m=\{(b-a)i +bj \pmod{m} : 0\leq i < d_1, 0\leq j< e_1\},\] 
the $C_i$'s form a partition of $\Z_m$ into $d_1$ cycles, each with $e_1$ elements.

\junk{
\begin{align*}
C_0&=\{0,b,2b,\ldots,(e_1-1)b\},\\
C_1&=\{(b-a),(b-a)+b,\ldots,(b-a)+(e_1-1)b\},\\
&\vdots\\
C_{d_1-1}&=\{(d_1-1)(b-a),\ldots,(d_1-1)(b-a)+(e_1-1)b\}.\\
\end{align*}
}

Let $c_{i,j}$ denote the $j$th element of $C_i$, i.e. $c_{i,j}=i(b-a)+jb \pmod{m}$. Note that any translate of $S$ contains two consecutive elements of two consecutive cycles, i.e. any translate of $S$ has the form $\{c_{i,j}, c_{i,j+1}, c_{i+1,j}, c_{i+1,j+1}\}$, where the first entry in the subscript is taken$\mod{d_1}$ and the second entry is taken$\mod{e_1}$. We describe an $S$-polychromatic 3-coloring for each of four subcases: (2a) $e_1$ is even, (2b) $d_1$ is even and $e_1$ is odd, (2c) $d_1$ and $e_1$ are both odd, with $e_1\leq 17$, and (2d) $d_1$ and $e_1$ are both odd, with $e_1\geq 19$.

\textbf{Subcase (2a):} Suppose $e_1$ is even. For $i=0,\ldots,\lfloor d_1/2\rfloor -1$, color each $C_{2i}$ by  $01010\ldots 01$ and each $C_{2i+1}$ by $02020\ldots 02$. Finally, if $d_1$ is odd, color $C_{d_1-1}$ by $1212\ldots 12$.

\textbf{Subcase (2b):} Suppose $d_1$ is even and $e_1$ is odd. For $i=0,\ldots,d_1/2 -1$, color each $C_{2i}$ by  $01010\ldots 011$ and each $C_{2i+1}$ by $22020\ldots 02$.

\textbf{Subcase (2c):} Suppose $d_1$ and $e_1$ are both odd, with $e_1\leq 17$. Since $e_1e_2\geq m>c \ge 289$, one of $e_1$ and $e_2$ is larger than 17, so $e_2> e_1$ and hence $d_1 > d_2$.  Since $d_1$ is the smaller of $d_1$ and $d_2$ that is not a multiple of 3, $d_2$ must be a multiple of 3, and thus so is $e_1$.

We color each $C_i$ with one of three patterns: $012012\ldots012$, $120120\ldots 120$, or $201201\ldots201$. Such a coloring is $S$-polychromatic so long as for all $i$, $C_i$ and $C_{i+1}$ are colored with different patterns. For $0 \le i \le (d_1-3)/2$, color $C_{2i}$ with the first pattern and color $C_{2i+1}$ with the second pattern. Finally, color $C_{d_1-1}$ with the third pattern.
 
\textbf{Subcase (2d):} Suppose $d_1$ and $e_1$ are both odd, with $e_1\geq 19$. Since $d_1$ is not divisible by 3 and $\min\{d_1,d_2\}>1$, $d_1\geq 5$.  Let  $e_1=u+v+w$ be a sum of odd integers $u$, $v$, $w$ with $u\geq v\geq w\geq u-2$.  Color $C_0$ in intervals of size $u,v,w$, using the patterns $0101\ldots010$ then $1212\ldots 121$ and then $2020\ldots 202$. For each  $i\geq 1$, color $C_i$ by taking a ``counterclockwise rotation'' of length $r_i$ of the coloring of $C_{i-1}$, so that the color of $c_{i,j+r}$ is the same as the color of $c_{i-1,j}$. For $1 \le i \le d_1-1$, if $u\leq r_i \leq v+w=e_1-u$, then each translate of $S$ meeting $C_{i-1}$ and $C_{i}$ receives all 3 colors. 

It remains to show that there are choices of $r_1, \ldots, r_{d_1-1}$ with $u\leq r_i \leq v+w=e_1-u$ so that of the translates of $S$ meeting $C_{d_1-1}$ and $C_0$ receive all three colors. The coloring of $C_0$ is a ``clockwise rotation'' of length $R= -r_1-r_2 -\cdots -r_{d_1-1}$ of the coloring of $C_{d_1-1}$, i.e. the  color of $c_{0,j-R}$ is the same as the color of $c_{d_1-1,j}$. Since for each $i$, $u\leq r_i \leq v+w=e_1-u$, it suffices to show that there is a multiple of $e_1$ in the interval $[d_1u,d_1(e_1-u)]$, ensuring there are choices for the $r_i$'s such that $R$ is congruent to a number between $u$ and $e_1-u$ $\pmod{e_1}$. This certainly holds if $d_1(e_1-2u)\geq e_1-1$ which, since $d_1\geq 5$, holds if $4e_1\geq 10u-1$. This inequality is true for $e_1\geq 19$.

This completes the multiple cycles case and the proof. 
\end{proofof}

\section{Colorings, Blocking Sets, Coverings, and Tilings}\label{polcov}

In this section we prove the results necessary to resolve Newman's conjecture. The key insight in proving Lemma~\ref{csps} is that the elements of a given color in an $S$-polychromatic coloring form a blocking set for $S$. While it is possible for $\alpha(S)$ to be equal to $1/p(S)$ (e.g. if $|S|=2$ then $\alpha(S)= 1/2 = 1/p(S)$), in general these two quantities are not equal.  For example, $p(\{0,1,3\}) = 2$, but by Lemma~\ref{alphc} and Theorem~\ref{newt}, $\alpha(\{0,1,3\}) = 2/5 < 1/2$.

\begin{proofof}{Lemma~\ref{csps}}
Let $\chi$ be an $S$-polychromatic coloring of $\Z$ with $p(S)$ colors.  Suppose $d \in \Z$ is greater than the diameter of $S$ and let $I_j = \{n\in \Z : jd \le n < (j+1)d\}$.  By the pigeonhole principle, for some $0 \le j_1 < j_2 \le (p(S))^d$ the coloring of the intervals $I_{j_1}$ and $I_{j_2}$ are identical, i.e. for $0\le k <d$, $\chi(j_1d+k) = \chi(j_2d+k)$.  Let $m = (j_2-j_1)d$. For any $n\in \Z$, denote by $r$ the remainder when $n$ is divided by $m$, so $0\le r < m$ . Let $\chi'$ be the coloring of $\Z$ where $\chi'(n) = \chi(j_1d+r)$.  Note that $\chi'$  uses $p(S)$ colors and is periodic with period $m$, i.e. for all $n \in \Z$, $\chi(n) = \chi(n +m)$. Furthermore, the coloring under $\chi'$ of any $d$ consecutive integers is identical to the coloring under $\chi$ of some $d$ consecutive integers, so $\chi'$ is $S$-polychromatic. Let $T_i = \{n\in \Z: \chi'(n) = i\}$. Since any periodic set has a defined density, $d(T_i)$ is defined for each $i$, and $\sum_{i=1}^{p(S)} d(T_i) = 1$.  Since $\chi'$ is $S$-polychromatic, for each $i$, each translate of $S$ contains an element of $T_i$, i.e. $T_i$ is also a blocking set for $S$. Thus for some $i$, $T_i$ is a blocking set for $S$ with density at most $1/p(S)$, which implies that $\alpha(S) \le 1/p(S)$. 
\end{proofof}

For any subset $T$ of an abelian group $G$, let $-T$ denote the set $\{-t: t\in T\}$. Lemma~\ref{CompTur} is well-known (see e.g. \cite{St86}) but for completeness we present a proof.

\begin{lemma}\label{CompTur}
Let $G$ be an abelian group, and $S \subseteq G$.  Then $T\subseteq G$ is a complement set for $S$ if and only if $-T$ is a blocking set for $S$.
\end{lemma}

\begin{proof}
Suppose $T$ is a complement set for $S$. For any $n \in G$, $-n \in S+T$, so there must be some $t\in T, s\in S$ such that $t+s = -n$. This implies $t=-n-s$, so $-n-s \in T$, and $n+s \in -T$. Thus  for every $n$, some element of $n+S$ is in $-T$, and $-T$ is a blocking set for $S$.

Conversely, suppose $-T$ is a blocking set for $S$. For the sake of contradiction, assume $T$ is not a complement set for $S$, i.e. there is some $-n\in G$ such that $-n \notin S+T$. This implies that for all $s \in S$, $-n-s \notin T$, which means for all $s\in S$, $n+s \notin -T$. Thus $n+S \subseteq G\setminus -T$, and so $-T$ is not a blocking set for $S$, a contradiction.
\end{proof}

\begin{proofof}{Lemma~\ref{alphc}}
Lemma~\ref{CompTur} implies that $T$ is a complement set for $S$ if and only if $-T $ is a blocking set for $S$. If they exist, the densities of $T$ and $-T$ are the same.
\end{proofof}

\subsection{Polychromatic Colorings and Tilings}\label{tiliff}

We now describe some relationships between polychromatic colorings and tilings.

\begin{theorem}\label{polytile}
Let $G$ be any abelian group. A finite set $S \subseteq G$ tiles $G$ by translation if and only if $p(S)=|S|$. Moreover, if $\chi$ is an $S$-polychromatic coloring of $G$ with $|S|$ colors and $T$ is the set of  elements of $G$ colored by $\chi$ with any given color, then $S \oplus T = G$.
\end{theorem}
\begin{proof}
Let $S=\{s_1, s_2, \ldots, s_k\}$, and suppose $S$ tiles $G$ with complement set $T \subseteq G$.  For each $n\in G$, define a coloring $\chi$ on $G$ so that $\chi(n) = i$ if $n =s_i+t$ for some $t\in T$. By the definition of tiling, this coloring is well-defined. For the sake of contradiction, assume $\chi$ is not $S$-polychromatic. Then for some $l$ where $1 \le l \le k$, there exists $n\in G$ and $s_i, s_j \in S$ with $i \neq j$ such that  $\chi(n+s_i)= \chi(n+s_j) = l$. Then there exist $t_1, t_2 \in T$, $t_1 \neq t_2$, such that $n+s_i = t_1 + s_l$ and $n+s_j = t_2 + s_l$. Subtracting these equations, we find that $s_i -s_j = t_1-t_2$.  Thus $t_2 + s_i=t_1 + s_j$, which is a contradiction.

Conversely, let $S=\{s_1, s_2, \ldots, s_k\}$, suppose $p(S)=|S|$, and let $\chi$ be an $S$-polychromatic coloring of $G$ with $|S|$ colors. Then for all $n \in G$, if $i \neq j$ then $\chi(n+s_i) \neq \chi(n+s_j)$. Let $T\subseteq G$ be the set of elements colored with a given color.  We show that $S \oplus T = G$. First assume for the sake of contradiction that two translates of $S$ share an element, i.e. there exist $s_i, s_j \in S$, $i \neq j$, $t_1, t_2 \in T$, $t_1 \neq t_2$, such that $s_i + t_1 = s_j + t_2$.  Let $n=t_1-s_j=t_2-s_i$, so $t_1 = n+s_j$ and $t_2 = n+s_i$.  Since $\chi(t_1)= \chi(t_2)$ we get $\chi(n+s_j)= \chi(n+s_i)$, so two elements of $n+S$ are colored identically, which is a contradiction. 

It remains to show that $S+T=G$. Suppose there is some $n \in G$ such that $n \notin S+T$. Then for all $i$, $n-s_i \notin T$, which implies that the $|S|$ elements of $n-S$ are colored with at most $|S|-1$ colors, i.e. two are colored identically.  Suppose $\chi(n-s_i) =\chi(n-s_j)$, where $i \neq j$.  Let $m=n-s_j-s_i$.  Then $m+S$ contains both $m+s_i =n-s_j$ and $m+s_j= n-s_i$. Since these integers are colored identically, $m+S$ is a translate of $S$ that does not contain all colors, which is a contradiction.
\end{proof}

Sets of integers with cardinality $n=3$ or 4 always have polychromatic number $n$ or $n-1$, and a corollary of Theorem~\ref{polytile} is that they have polychromatic number $n-1$ if and only if they do not tile $\Z$. According to Remark 5.6 in \cite{BJR11}, $c(\{0,1,3,4,8\}) = 3/11 > 1/4$.  Thus by Lemma~\ref{csps}, $\{0,1,3,4,8\}$ is an example of a set with cardinality 5 and polychromatic number 3. The results of \cite{AKN95} and \cite{HS16} imply that for sets $S$ with large cardinality $n$ the cardinality and polychromatic number  of $S$ can differ by a factor of $1/\ln n$. 

We now state some other corollaries of Theorem~\ref{polytile}.

\begin{corollary}
If a finite set $S$ tiles an abelian group $G$ by translation, then any $S$-polychromatic coloring of $G$ with $|S|$ colors is also a $(-S)$-polychromatic coloring.
\end{corollary}

\begin{proof}
Suppose $S$ tiles $G$.  By Theorem~\ref{polytile}, there exists an $S$-polychromatic coloring $\chi$ of $G$ with $|S|$ colors.  Let $T \subseteq G$ be the set of all elements of a given color. Again by Theorem~\ref{polytile}, $S + T = G$. Therefore by Lemma~\ref{CompTur}, $-T$ is a blocking set for $S$, i.e. for all $n \in G$, $n+S \nsubseteq G\setminus (-T)$. This implies that for all $n \in G$, $-n-S \nsubseteq G\setminus T$, i.e. $T$ is a blocking set for $-S$.  Since $T$ is a blocking set for $-S$ for every color choice, every translate of $-S$ contains every color, i.e. the coloring $\chi$ is $(-S)$-polychromatic.
\end{proof}

Define $t(S)$ to be the cardinality of the largest subset of $S$ that tiles $G$. 

\begin{corollary}
For any finite subset $S$ of an abelian group $G$, $p(S) \ge t(S)$.
\end{corollary}

If $S\subseteq \Z$, $|S|\le 3$, then $p(S)=t(S)$. But these parameters can be different for sets of integers with at least four elements.  For example, $S=\{0,1,3,7\}$ is an example of a set where $t(S)=2$, but $p(S)=3$.

\begin{question}
For sets $S$ of a given cardinality, how large can the gap between $t(S)$ and $p(S)$ be?
\end{question}

\section{Polychromatic Colorings in $\Z^d$}\label{z2}

In this section we consider polychromatic numbers in the case where $G=\Z^d$, $d\ge 2$. We will frequently ``project'' a set $S \subseteq \Z^d$ to another set $S'\subseteq \Z^{d-1}$ as follows. Let $d \ge 2$, and $\vect{w}=(w_1,\ldots, w_{d-1},1) \in \Z^d$. Define $f_{\vect{w}}: \Z^d \rightarrow \Z^{d-1}$ so that if $\vect{s}=(v_1, \ldots, v_d) \in \Z^d$, 
\[f_{\vect{w}}(\vect{s})= (v_1, \ldots, v_{d-1}) - v_d (w_1,\ldots, w_{d-1}).\]
We call $f_{\vect{w}}(\vect{s})$ the \textit{projection} of $\vect{s}$ along $\vect{w}$. Given a set $S \subseteq \Z^d$, we call the set $f_{\vect{w}}(S) \subseteq \Z^{d-1}$ the \textit{projection} of $S$ along $\vect{w}$.


For example, if $\vect{s}=(2,7,4)$ and $\vect{w} = (3,1,1)$, the projection of $\vect{s}$ along $\vect{w}$ is $f_{\vect{w}}(\vect{s}) = (2,7) - 4(3,1)= (-10,3)$. As another example, note that if $\vect{s}=(v_1, \ldots, v_d) \in \Z^d$, the vector $\vect{s}' = (v_1, \ldots, v_{d-1}) \in \Z^{d-1}$ is the projection of $\vect{s}$ along $\vect{w}= (0,\ldots, 0,1)$.

\begin{lemma}\label{proj}
Let $d \ge 2$, and $\vect{w}=(w_1,\ldots, w_{d-1},1) \in \Z^d$. Let $S \subseteq \Z^d$, and suppose $S' \subseteq \Z^{d-1}$ is the projection of $S$ along $\vect{w}$.  Then $p_{\Z^d}(S) \ge p_{\Z^{d-1}}(S')$.
\end{lemma}
\begin{proof}
Since $f_{\vect{w}}: \Z^d \to \Z^{d-1}$ is a homomorphism, the result follows from Lemma~\ref{hom}.
\end{proof}

\begin{proposition}\label{fullcard}
Let $d \ge 2$. For any $S\subseteq \Z^d$, there is a projection $S' \subseteq \Z^{d-1}$ where $|S|=|S'|$.
\end{proposition}

\begin{proof}
Let $S =\{\vect{s}_1, \ldots, \vect{s}_k\} \subseteq \Z^d$ and suppose $\vect{w} = (w_1, \ldots, w_{d-1},1) \in \Z^d$. For $1 \le i \le k$ let $\vect{s}_i'=f_{\vect{w}}(\vect{s}_i)$. For $1 \le i \le k$, let $s_{id}$ denote the last coordinate of $\vect{s}_i$ and note that if $i \neq j$, $\vect{s}_i' = \vect{s}_j'$ if and only if 
\[\vect{w} = \frac{1}{s_{id}-s_{jd}}(\vect{s}_i - \vect{s}_j).\]
In other words $\vect{s}_i' = \vect{s}_j'$ if and only if $\vect{w}$ is parallel to $\vect{s}_i - \vect{s}_j$.  Since the number of differences $\vect{s}_i - \vect{s}_j$ is finite, we can choose $\vect{w}$ so that it is not parallel to any of these. For this choice of $\vect{w}$, for all $1\le i\neq j \le k$, $\vect{s}_i' \neq \vect{s}_j'$ Then $S' = \{\vect{s}_1', \ldots, \vect{s}_k'\}$, is a projection $S$ with $|S'| = |S|$.
\end{proof}

\begin{theorem}\label{hsd}
Fix $d \ge 2$. For a finite set $S \subseteq \Z^d$, $p(S) \ge \frac{(1+o(1))|S|}{\ln|S|}$.
\end{theorem}

\begin{proof}
Given $S \subseteq \Z^d$, Proposition~\ref{fullcard} implies we can project $d-1$ times to ultimately obtain a set $S' \subseteq \Z$, with $|S'|=|S|$.  Theorem~\ref{Harris}, along with repeated application of Lemma~\ref{proj}, implies $p_{\Z^d}(S) \ge p_{\Z}(S') \ge \frac{(1+o(1))|S'|}{\ln|S'|} = \frac{(1+o(1))|S|}{\ln|S|}$.
\end{proof}

We will be interested in the case where $S= \{\vect{s}_0, \vect{s}_1, \vect{s}_2, \ldots, \vect{s}_k\}\subseteq \Z^d$ contains a set of collinear points.  In the general case, we can assume for each $i$ that $\vect{s}_i = (b_1 + l_ia_1, b_2+ l_i a_2, \ldots, b_d + l_ia_d)$, where $0=l_0 <l_1 < l_2 < \cdots <l_k$, $a_i, b_i \in \Z$, and $\gcd(a_1, a_2, \ldots, a_d)=1$.  However since translation does not affect the polychromatic number, we will restrict our attention to the case where $a_1 > 0$ and for all $i$, $b_i=0$.

\begin{theorem}\label{pcolzd}
Let $d \ge 2$. Let $S = \{\vect{s}_0, \vect{s}_1, \vect{s}_2, \ldots, \vect{s}_k\}$ be a set of $k+1$ collinear points in $\Z^d$ where for each $i$, $\vect{s}_i = (l_ia_1, l_i a_2, \ldots, l_ia_d)$, where $0=l_0 <l_1 < l_2 < \cdots <l_k$, $a_i \in \Z$, $a_1 > 0$, and $\gcd(a_1, a_2, \ldots, a_d)=1$.  Let $S' = \{0, l_1, l_2, \ldots, l_k\} \subseteq \Z$.  Then $p_{\Z^d}(S) = p_{\Z}(S')$.
\end{theorem}

\begin{proof}
Let $S'' =  \{0, l_1a_1, l_2a_1, \ldots, l_ka_1\} \subseteq \Z$.  By Lemma~\ref{reduc}, Part (i), $p_{\Z}(S') = p_{\Z}(S'')$.  Since $S''$ can be obtained from $S$ by a sequence of $d-1$ projections, Lemma~\ref{proj} implies $p_{\Z^d}(S) \ge p_{\Z}(S'')= p_{\Z}(S')$.  For the other direction, let $\vect{a} = (a_1,\ldots, a_d) \in \Z^d$ and note that the function $\phi : \Z \to \Z^d$ where $\phi(n) = n\vect{a}$ is a homomorphism where $\phi(S') = S$.  Thus by Lemma~\ref{hom},  $p_{\Z}(S') \ge p_{\Z^d}(S)$
%
\end{proof}

Now we return to the subject of tilings.  Lemma~\ref{subgroup} and Theorems~\ref{startile}, \ref{dcro}, and \ref{1tiled} are well-known in the field of discrete geometry (see e.g. Section III of \cite{St86}) as simple examples of ``splitting'' groups. We restate them here using the language of polychromatic colorings.

\begin{lemma}\label{subgroup}
If a set $S\subseteq G$ tiles a nontrivial subgroup $H$ of $G$, then $S$ tiles $G$.
\end{lemma}

\begin{proof}
Suppose $S\subseteq G$ tiles a nontrivial subgroup $H$ of $G$. Theorem~\ref{polytile} implies $p_H(S) = |S|$, so by Lemma~\ref{ambient}, $p_G(S) = |S|$.  By Theorem~\ref{polytile}, $S$ tiles $G$.
\end{proof}

For any $d\ge 1$, let $\vect{0}$ denote the element $(0,0,\ldots, 0)\in \Z^d$ and
 let $\vect{e}_i$ denote the element $(0,\ldots, 0,1,0, \ldots,0) \in \Z^d$ with all 0's except for a 1 in the $i$th position. For $\vect{s} = (v_1, \ldots, v_d) \in \Z^d$, let $-\vect{s}=(-v_1, \ldots, -v_d)$.  Define the $d$-semicross $SC_d = \{\vect{0},\vect{e}_1, \ldots, \vect{e}_d\}$ and the $d$-cross $C_d=\{\vect{0},\vect{e}_1,-\vect{e}_1, \vect{e}_2, -\vect{e}_2, \ldots, \vect{e}_d, -\vect{e}_d\}$.  Theorem~\ref{polytile} implies that any finite set $S \subseteq G$ with $p(S)=|S|$ tiles $G$, and we use this insight to show that these sets tile $\Z^d$.

\begin{theorem}\label{startile}
For all $d \ge 1$, the $d$-semicross $SC_d=\{\vect{0},\vect{e}_1, \ldots, \vect{e}_d\}$ tiles $\Z^d$. 
\end{theorem}

\begin{proof}
Consider the coloring $\chi:\Z^d \to [d+1]$ where $\chi(v_1, \ldots, v_d) = v_1 + 2v_2 + 3v_3 + \cdots +dv_d \pmod{d+1}$.  On any translate $\vect{n} + SC_d \subseteq \Z^d$, the colors $\chi(\vect{n}+\vect{0}), \chi(\vect{n}+\vect{e}_1), \chi(\vect{n}+\vect{e}_2), \ldots, \chi(\vect{n}+\vect{e}_d)$ are $\chi(\vect{n}), \chi(\vect{n})+1, \chi(\vect{n})+2, \ldots, \chi(\vect{n})+d$ $\pmod{d+1}$.  They are all different, so $\chi$ is $SC_d$-polychromatic with $|SC_d|=d+1$ colors.  By Theorem~\ref{polytile}, $SC_d$ tiles $\Z^d$.
\end{proof}

\begin{theorem}\label{dcro}
For all $d \ge 1$, the $d$-cross $C_d=\{\vect{0},\vect{e}_1,-\vect{e}_1, \vect{e}_2, -\vect{e}_2, \ldots, \vect{e}_d, -\vect{e}_d\}$ tiles $\Z^d$.
\end{theorem}

\begin{proof}
The $(2d+1)$-coloring $\chi:\Z^d \to [2d+1]$ where $\chi(v_1, \ldots, v_d) = v_1 + 2v_2 + 3v_3 + \cdots +dv_d \pmod{2d+1}$ is $C_d$-polychromatic:  On any translate $\vect{n} + C_d \subseteq \Z^d$, the colors $\chi(\vect{n}+\vect{0}), \chi(\vect{n}+\vect{e}_1), \chi(\vect{n}-\vect{e}_1), \chi(\vect{n}+\vect{e}_2), \chi(\vect{n}-\vect{e}_2),\ldots, \chi(\vect{n}+\vect{e}_d), \chi(\vect{n}-\vect{e}_d)$ are $\chi(\vect{n}), \chi(\vect{n})+1, \chi(\vect{n})-1, \chi(\vect{n})+2, \chi(\vect{n})-2, \ldots, \chi(\vect{n})+d, \chi(\vect{n})-d$ $\pmod{2d+1}$. 
\end{proof}

\begin{theorem}\label{1tiled}
Let $d \ge  2$. Let $S\subseteq \Z^d$ be a set that contains $\vect{0}$ and $j \le d$ other elements $\vect{s}_1, \ldots, \vect{s}_j$, where no nontrivial integer linear combination of $\{\vect{s}_1, \ldots, \vect{s}_j\}$ is $\vect{0}$. Then $S$ tiles $\Z^d$.
\end{theorem}

\begin{proof}
Let $H\subseteq \Z^d$ be the set of all integer linear combinations of $\{\vect{s}_1, \ldots, \vect{s}_j\}$. By Theorem~\ref{startile}, there is a set $T\subseteq \Z^j$ such that $\{\vect{0},\vect{e}_1, \ldots, \vect{e}_j\} \oplus T = \Z^j$.  Let $M: \Z^j \rightarrow \Z^d$ be the unique linear transformation which maps $\vect{e}_i$ to $\vect{s}_i$ for each $i \le j$. Then $\{\vect{0}, \vect{s}_1, \ldots, \vect{s}_j\}$ tiles $H$ with complement set $\{M(t) : t \in T\}$.   Since $H$ is a subgroup of $\Z^d$, by Lemma~\ref{subgroup}, $S$ tiles $\Z^d$.
\end{proof}

We can now determine the polychromatic number of any set $S$ of cardinality 3 or 4 in $\Z^d$, $d \ge 2$. Since translation does not affect polychromatic numbers, in all cases we may assume $\vect{0} \in S$.

\begin{theorem}
Let $d \ge 2$ and suppose $S \subseteq \Z^d$ has cardinality 3, with $\vect{0} \in S$. Then $p_{\Z^d}(S) = 3$ if the three points are in general position or if they are collinear and there exists $S' \subseteq \Z$ with $p_{\Z}(S') = 3$ such that $S'$ is the image of $S$ after $d-1$ projections. Otherwise $p_{\Z^d}(S) = 2$.
\end{theorem}

\begin{proof}Theorem~\ref{1tiled} implies that if $d \ge 2$ and $S \subseteq \Z^d$ consists of three points in general position, then $S$ tiles $\Z^d$, and thus $p(S)=3$.  If $S \subseteq \Z^d$ has three collinear points, then Theorem~\ref{pcolzd} implies the problem is equivalent to finding the polychromatic number of a set of three integers, which is either 2 or 3 and can be determined using Theorem~\ref{newtile77}.
\end{proof}

\begin{theorem}
Let $d \ge 2$ and suppose $S \subseteq \Z^d$ has cardinality $4$, with $\vect{0} \in S$. Then
\begin{enumerate}
\item If all points of $S$ are collinear, $p_{\Z^d}(S)$ is $3$ or $4$.
\item If exactly three points of $S$ are collinear, $p_{\Z^d}(S)=4$.
\item If $d\ge 3$ and $S$ has four points in general position, $p_{\Z^d}(S)=4$.
\item If $d=2$ and $S$ has four points in general position, $p_{\Z^2}(S)$ is $3$ or $4$.
\end{enumerate}
\end{theorem}

\begin{proof}
 For $d \ge 2$ and a set $S\subseteq \Z^d$ with $|S|=4$,  Proposition~\ref{fullcard} implies that there is a set $S' \subseteq \Z$ where $|S'|=4$ and $S'$ can be obtained by $d-1$ projections of $S$. Thus Theorem~\ref{ps34} and Lemma~\ref{proj} imply that $p(S)\ge 3$. Determining whether $p(S)$ is 3 or 4 is equivalent to determining whether $S$ tiles $\Z^d$.  As with the $|S|=3$ case, we can examine cases depending on how many points of $S$ are collinear.

If the four points of $S$ are in general position, then if none is a nontrivial integer linear combination of the others, $p(S)=4$ by Theorem~\ref{1tiled}.  Otherwise, we can assume $S\subseteq \Z^2$.  In this case, $p(S)$ can be 3, for example if $S=\{(0,0), (1,0), (0,1), (1,2)\} \subseteq \Z^2$. It can also be 4, for example if $S=\{(0,0), (1,0), (0,1), (1,1)\} \subseteq \Z^2$.  Szegedy~\cite{Sz01} gave an algorithm to determine if a set of cardinality 4 tiles $\Z^2$.

If the four points of $S$ are all collinear, then $p(S)$ is determined by applying Theorems~\ref{pcolzd} and \ref{newtile77}.

If exactly three of the four points are collinear, then without loss of generality assume $S= \{\vect{0}, (a,0, \ldots, 0), (b,0,\ldots 0), \vect{s}\}$, with $0<a<b$ and $\vect{s} = (s_1, \ldots, s_d)$. Then $S$ is in the subgroup $\{x \in \Z^d: x_i \in s_i\Z \text{ for } 2 \le i \le d\}$.  By Lemmas~\ref{hom} and \ref{ambient} we may assume that $s_i=1$ for some $2 \le i \le d$, and thus by a sequence of projections, $S$ can be projected to the set $\{0, a, b, c\}$, for any $c \in \Z$. By Lemma~\ref{proj} it suffices to show that there exists $c\in \Z$ such that $p_{\Z}(\{0, a, b, c\})=4$. Without loss of generality, we may assume that $a$ and $b$ have different parity, and in this case Theorem~\ref{newtile77} implies that $S=\{0,a,b,a+b\}$ has polychromatic number 4.
\end{proof}

The fact that $p_{\Z^d}(S)=4$ if $S$ contains exactly three collinear points implies that for any set $S$ of three integers, there is a 4-coloring of $\Z$ so that every translate of $S$ gets three different colors. Here is an explicit example of one such coloring. Without loss of generality we need only consider sets of the following form: Let $S=\{0,a,b\} \subseteq \Z$ where $a$ and $b$ are positive with $a$ even and $b$ odd (note that we do not specify which is larger). Define the \textit{alternating block 4-coloring relative to $S$} as follows: Given any $m \in \Z$, let $q_m$ and $r_m$ be the unique integers such that $m=2aq_m+r_m$, where $-a\le r_m <a$. Let $X(m)=0$ if $r_m \ge 0$, $X(m)=1$ otherwise.  Let $Y(m)=0$ if $m$ is even, $Y(m)=1$ otherwise.  Define $\chi$, the alternating block 4-coloring relative to $S$, so that $\chi(m) = (X(m),Y(m))$. 

\begin{theorem}
Let  $S=\{0,a,b\} \subseteq \Z$ with $a, b >0$, $a$ even, and $b$ odd. If the integers are colored with the alternating block 4-coloring relative to $S$ then every translate of $S$ has elements of three different colors.
\end{theorem}
\begin{proof}
For any translate $n+S = \{n, n+a, n+b\}$ of $S$, $X(n)\neq X(n+a)$, while $Y(n) = Y(n+a) \neq Y(n+b)$. Thus $\chi$ has the property that any translate of $S$ contains elements with three different colors.
\end{proof}



Given a set of three integers, the alternating block 4-coloring shows that there is a 4-coloring of the integers so that every translate gets three different colors. If $S\subset \Z$, $|S|=4$, is there a 5-coloring of $\Z$ so that every translate of $S$ has 4 colors? More generally, we ask the following question.

\begin{question}
Let $d \ge 1$. Given $k,n \in \Z$ with $k \le n$, let $p(n,k)$ denote the minimum $r$ so that any $S \subseteq \Z$ with $|S|=n$ has an $r$-coloring where every translate of $S$ gets at least $k$ colors.  What is an asymptotic upper bound on $p(n,k(n))$ for natural choices of $k(n)$? 
\end{question}

\junk{
\begin{proposition}\label{3noncol}
Let $S= \{(0,0), (a, b), (c,d)\} \subseteq \Z^2$.  If the three points in $S$ are not collinear, then $p(S) = 3$.
\end{proposition}

\begin{proof} 3-color the line containing two of the points with a special block coloring.  Then color the parallel line through the third point s.t. $S$ has all three colors. By the properties of the special block coloring, the second line will be colored with a permutation of the coloring of the first line.  Then use these two lines as a template for the rest of the plane.
\end{proof}
}

\section{Acknowledgement}
This research was conducted at a workshop made possible by the Alliance for Building Faculty Diversity in the Mathematical Sciences (DMS 0946431), held at the Institute for Computational and Experimental Research in Mathematics.  

The authors wish to thank the anonymous referee for generalizing some results and greatly improving the exposition.

\end{document}